\documentclass[twoside,11pt]{article}

 \usepackage[abbrvbib, preprint]{jmlr2e}
\usepackage[linesnumbered,ruled]{algorithm2e}
\usepackage{tikz}

\usepackage{mathtools}
\usetikzlibrary{shapes.geometric, arrows}
\usepackage{algorithmic}
\usepackage[english]{babel}
\usepackage[utf8]{inputenc}
\usepackage{subfig}
\usepackage{afterpage} 
\usepackage{mathtools}
\usepackage{amsfonts}
\usepackage[colorinlistoftodos]{todonotes}
\usepackage{enumerate}
\usepackage{hyperref}
\newcommand{\excise}[1]{}

\newcommand\<{\langle}

\newcommand\NN{\mathbb{N}}

\newcommand\RR{\mathbb{R}}

\newcommand\ZZ{\mathbb{Z}}
\newcommand\EE{\mathbb{E}}

\newcommand{\T}{\mathrm{{\scriptstyle T}}}
\newcommand\MM{\mathcal{M}}

\renewcommand\>{\rangle}

\DeclareMathOperator{\eig}{eig}
\DeclarePairedDelimiter\floor{\lfloor}{\rfloor}

\usepackage{changepage}
\newif\ifhighlightchanges

\newcommand{\changedreviewerone}[1]{\ifhighlightchanges\textcolor{red}{#1}\else#1\fi}

\newcommand{\changedreviewertwo}[1]{\ifhighlightchanges\textcolor{blue}{#1}\else#1\fi}

\newcommand{\changedrevieweronetwo}[1]{\ifhighlightchanges\textcolor{orange}{#1}\else#1\fi}

\usepackage{lastpage}
\jmlrheading{24}{2023}{1-\pageref{LastPage}}{5/22; Revised
11/22}{3/23}{22-0503}{Didong Li, Wenpin Tang and Sudipto Banerjee}
\ShortHeadings{Gaussian Processes on Compact Riemannian Manifolds}{Li, Tang and Banerjee}
\firstpageno{1}

\begin{document}

\title{Inference for Gaussian Processes with Mat\'ern Covariogram on Compact Riemannian Manifolds}

\author{\name Didong Li \email didongli@unc.edu \\
       \addr Department of Biostatistics\\
       University of North Carolina at Chapel Hill\\
       Chapel Hill, NC 27599, USA
       \AND
       \name Wenpin Tang \email wt2319@columbia.edu \\
       \addr Department of Industrial Engineering and Operations Research,\\
       Columbia University \\
      New York, NY 10027, USA
		\AND
		\name Sudipto Banerjee \email sudipto@ucla.edu\\
		\addr   Department of Biostatistics\\
		University of California, Los Angeles\\
		Los Angeles, CA 90095 USA
}

\editor{Marc Peter Deisenroth}

\maketitle

\begin{abstract}
Gaussian processes are widely employed as versatile modelling and predictive tools in spatial statistics, functional data analysis, computer modelling and diverse applications of machine learning. They have been widely studied over Euclidean spaces, where they are specified using covariance functions or covariograms for modelling complex dependencies. There is a growing literature on Gaussian processes over Riemannian manifolds in order to develop richer and more flexible inferential frameworks for non-Euclidean data. While numerical approximations through graph representations have been well studied for the Mat\'ern covariogram and heat kernel, the behaviour of asymptotic inference on the parameters of the covariogram has received relatively scant attention. We focus on asymptotic behaviour for Gaussian processes constructed over compact Riemannian manifolds. Building upon a recently introduced Mat\'ern covariogram on a compact Riemannian manifold, we employ formal notions and conditions for the equivalence of two Mat\'ern Gaussian random measures on compact manifolds to derive \changedreviewertwo{the parameter that is identifiable, also known as the microergodic parameter,} and formally establish the consistency of the maximum likelihood estimate and the asymptotic optimality of the best linear unbiased predictor. The circle is studied as a specific example of compact Riemannian manifolds with numerical experiments to illustrate and corroborate the theory.

\end{abstract}


\section{Introduction}

Gaussian processes are pervasive in spatial statistics, functional data analysis, computer modelling and machine learning applications because of the flexibility and richness they \changedreviewertwo{allow} in modelling complex dependencies \citep{rasm08, stein2012interpolation, gelfand2010handbook, cressie2011statistics, banerjee2014hierarchical}. For example, in spatial statistics Gaussian processes are widely used to model spatial dependencies in geostatistical models and perform spatial prediction or interpolation (``kriging'') \citep{matheron1963principles}. In non-parametric regression models Gaussian processes are used to model unknown functions and, specifically in Bayesian contexts, act as priors over functions \citep{ghosal2017fundamentals}. A typical modelling framework assumes $y(x) = \mu(x) + Z(x) + \epsilon(x)$ for inputs $x$ (e.g., spatial coordinates; functional inputs) over a domain ${\cal D}$, where $y(x)$ is a dependent variable of interest, $\mu(x)$ is a mean function, $Z(x)$ is a \changedreviewertwo{zero-mean} Gaussian process and $\epsilon(x)$ is a noise process\footnote{This article does not consider the noise process, which introduces additional difficulties that are beyond the scope of the current manusript; see \cite{tang2019identifiability} for related developments in Euclidean space.}. These frameworks can also be adapted to deal with discrete outcomes and applied to classification problems \citep{bernardo1998regression}. 
Gaussian processes are also being increasingly employed in deep learning and reinforcement learning \citep{damianou2013deep, deisenroth2013gaussian}. The current manuscript focuses upon inferential properties of $Z(x)$ when ${\cal D}$ is not necessarily Euclidean but a compact Riemannian manifold.

A Gaussian process is determined by its covariogram, also known as the covariance function. In Euclidean space, the Mat\'{e}rn covariogram \citep{matern86} is especially popular in spatial statistics and machine learning \citep[see, e.g.,][for an extensive discussion on the theoretical properties of the Mat\'ern covariogram]{stein2012interpolation}. A key attraction of the Mat\'{e}rn covariogram is the availability of a smoothness parameter for the process. Several simpler covariograms, such as the exponential, arise as special cases of the Mat\'ern. 

This article is motivated by the emergence of non-Euclidean data, especially manifold data, in a variety of scientific fields over the last decade. As a consequence, inference for Gaussian processes on manifolds have been attracting attention in spatial statistics and machine learning in settings where the data generating process is more appropriately modelled over non-Euclidean spaces. Taking climate science as an example, geographic data involving geopotential height, temperature and humidity are measured at global scales and are more appropriately treated as (partial) realisations of a spatial process over a sphere \citep[see, e.g.,][]{banerjee2005geodetic, junstein2008, jeong2015class}. Data arising over domains with irregular shapes or examples in biomedical imaging where the domain is a three-dimensional shape of an organ comprise other examples where inference for Gaussian processes over manifolds will be relevant \citep[see, e.g.,][and references therein]{gao2019gaussian}. Motivated by isotropic covariograms in Euclidean space, it is natural to replace Euclidean distance by an appropriate geodesic distance to define a ``Mat\'ern" covariogram on Riemannian manifolds. However, this formal generalisation \changedreviewertwo{is not valid for the squared exponential covariogram, or Mat\'ern with $\nu=\infty$ \citep{feragen2015geodesic}, unless the manifold is flat. For Mat\'ern with $\nu\in(1/2,\infty)$, this naive generalisation is not even valid on the sphere \citep{gneiting2013strictly}}. Recently, valid covariograms for smooth Gaussian processes on general Riemannian manifolds have been constructed based upon heat equations, Brownian motion and diffusion models on manifolds \citep{castillo2014thomas, niu2019intrinsic, dunson2020diffusion}. However, these covariograms lack flexibility, especially in terms of modelling smoothness.

\cite{whittle1963stochastic} proposed a new representation of GP by stochastic partial differential equations. Following this path, \cite{lindgren2011explicit} introduced a ``Mat\'ern" family on generic compact Riemannian manifolds with three parameters \changedreviewerone{involved} in the covariogram. Since such Mat\'ern covariograms involve the spectrum of the Laplace-Beltrami operator, a numerical approximation to the covariogram is needed for most nontrivial manifolds. There is a rich literature focusing on approximations to the covariogram using tools from harmonic analysis, graph Laplacians, and stochastic partial differential equations \citep{sanz2020spde, sanz2021finite}. However, the study of statistical inference for the parameters in the Mat\'ern covariogram remains relatively sparse. 


In Euclidean domains $\mathbb{R}^{d}$ with $d\leq 3$, while not all parameters in the Mat\'ern covariogram are consistently estimable within the paradigm of ``fixed-domain'' or ``in-fill'' asymptotic inference \citep[see, e.g.][]{stein2012interpolation, zhang2004inconsistent}, \changedreviewertwo{certain parameters, customarily referred to as microergodic parameters, which can identify Gaussian processes specified by Mat\'ern covariograms are consistently estimable} (see Section~\ref{sec: GP_Euclidean}). Furthermore, the maximum likelihood estimator of the spatial variance under any misspecified decay parameter is consistently and asymptotically normally distributed \citep{du2009fixed, kaufman2008covariance, wang2011fixed}, while predictive inference is also asymptotically optimal using maximum likelihood estimators \citep{kaufman2013role}. Recently, \cite{BFFP19} and \cite{MB19} considered more general classes of covariance functions outside of the Mat\'ern family and studied the consistency and asymptotic normality of the maximum likelihood estimator for the corresponding microergodic parameters.

Our current contribution develops asymptotic inference for a flexible and rich Mat\'ern-type covariogram on compact Riemannian manifolds. 
We review the Mat\'ern covariogram (Section~\ref{subsec: matern_compact_riemannian_manifold}) on general compact Riemannian manifolds from the perspective of stochastic partial differential equations with reasonably tractable covariograms and spectral densities \citep{NEURIPS2020_92bf5e62}. 
Our specific results emanate from a sufficient and necessary condition for the equivalence of two Gaussian random measures on compact Riemannian manifolds with Mat\'ern or squared exponential covariograms (Section~\ref{subsec: identifiability}). 
We subsequently establish (Section~\ref{subsec: consistency_mle}) that for Gaussian measures with Mat\'ern covariograms \changedreviewertwo{the smoothness parameter is identifiable, while the spatial variance and decay parameters are not identifiable when $d\leq 3$, where $d$ is the dimension of the manifold. For $d\geq 4$, all three parameters are identifiable}. 
For squared exponential covariograms on manifolds with arbitrary dimension, we show that \changedreviewertwo{both parameters are identifiable}. Again, this problem is still open in Euclidean spaces. For Mat\'ern covariograms on manifolds with $d\leq 3$, we formally establish that the maximum likelihood estimate of the spatial variance with a misspecified decay parameter is still consistent. Next, we turn to predictive inference (Section~\ref{subsec: prediction}) and 
show that for any misspecified decay parameter in the Mat\'ern covariogram, the best linear unbiased predictor derived from the maximum likelihood estimate is asymptotically optimal. Finally, for spheres with dimension less than $4$, we explicitly study the Mat\'ern covariogram, the microergodic parameter, the consistency of the maximum likelihood estimate and the optimality of the best linear unbiased predictor (Section~\ref{sec:sphere}).
Proofs and mathematical details surrounding our main results are provided in the Appendix. 

\section{Gaussian Processes in Euclidean spaces}\label{sec: GP_Euclidean}
Let $Z = \{Z(x) : x\in \MM\subset \RR^d\}$ be a zero-mean Gaussian process on a bounded domain $\MM$.  The process $Z(\cdot)$ is characterised by its covariogram $k(x,y) = \mathbb{E}(Z(x) Z(y))$, $x, y \in \MM$ so that for any finite collection of points, say $x_1,\cdots, x_n\in \MM$, we have $\left(Z(x_1),\cdots,Z(x_n)\right)^{\T}\sim \mathcal{N}(0,\Sigma)$, where $\Sigma$ is the $n\times n$ covariance matrix with $(i,j)$-th entry $\Sigma_{ij} = k(x_i, x_j)$. The Mat\'ern process is a zero-mean stationary Gaussian process specified by the covariogram\footnote{\cite{solin2019know} provides an alternative definition based on PDEs with boundary conditions.},
\begin{equation}
	\label{eq:materncovr}
	k(x,y)=\frac{\sigma^2\left(\alpha\|x-y\|\right)^{\nu}}{\Gamma(\nu)2^{\nu-1}}K_\nu\left(\alpha\|x-y\|\right),~~x,y\in \MM\subset\RR^d,
\end{equation}
where $\sigma^2 > 0$ is called the partial sill or spatial variance, $\alpha > 0$ is the {scale or decay parameter}, $\nu > 0$ is a smoothness parameter, $\Gamma(\cdot)$ is the Gamma function, and $K_{\nu}(\cdot)$ is the modified Bessel function of the second kind of order $\nu$ \cite[Section 10]{AS65}.
The Mat\'ern covariogram in \eqref{eq:materncovr} is isotropic and its spectral density \changedreviewertwo{(also known as the Hankel-Fourier transform, \cite{genton2001classes}) is given by}
$$f(u)=\frac{\sigma^2\alpha^{2\nu}}{\pi^{d/2}(\alpha^2+u^2)^{\nu+d/2}},~~u\geq 0.$$
\subsection{Identifiability}\label{subsec: identifiability_euclidean} Let $P_0$ and $P_1$ be Gaussian measures corresponding to Mat\'ern parameters  $\{\sigma_0^2,\alpha_0,\nu\}$ and $\{\sigma_1^2,\alpha_1,\nu\}$, respectively. Two measures are said to be equivalent, denoted by $P_0\equiv P_1$, if they are absolutely continuous with respect to each other. Two equivalent measures cannot be distinguished \changedreviewertwo{no matter how dense the observations are.} \cite{zhang2004inconsistent} showed that when $d<4$, $P_0$ is equivalent to $P_1$ if and only if $\sigma_0^2\alpha_0^{2\nu}=\sigma_1^2\alpha_1^{2\nu}$. Hence, $\sigma^2$ and $\alpha$ do not admit asymptotically consistent estimators, while $\sigma^2\alpha^{2\nu}$, also known as a \emph{microergodic parameter}, is consistently estimable. For $d > 4$, \cite{A10} proved that both $\sigma^2$ and $\alpha$ are consistently estimable. The case for $d = 4$ remains unresolved. The integral test offers a sufficient (but not necessary) condition on the spectral densities to determine whether two measures are equivalent. \changedreviewertwo{While unidentifiable parameters are never consistently estimable, identifiable parameters may be consistently estimable. However, deriving an explicit construction for such a consistent estimator is often challenging and is beyond the scope of the current manuscript; we identify this as an area of future research.}

\subsection{Parameter estimation}\label{subsec: parameter_estimation_euclidean}
In practice, the maximum likelihood estimate is customarily used to estimate unknown parameters in the covariogram. Let $L_n(\sigma^2,\alpha)$ be the likelihood function:
\begin{equation}\label{eqn:lkhd}
	L_n(\sigma^2,\alpha)=(2\pi\sigma^2)^{-n/2}\det(\Gamma_n(\alpha))^{-1/2}\exp\left\{-\frac{1}{2\sigma^2}Z_n^\T \Gamma_n(\alpha)^{-1}Z_n\right\}\;,
\end{equation}
where $Z_n=(Z(x_1),\cdots,Z(x_n))^{\T}$ and $(\Gamma_n(\alpha))_{i,j}=\frac{\left(\alpha\|x_i-x_j\|\right)^{\nu}}{\Gamma(\nu)2^{\nu-1}}K_\nu\left(\alpha\|x_i-x_j\|\right)$ is independent of $\sigma^2$. Given $\alpha$, the maximum likelihood estimation of $\sigma^2$ is given by \citep{stein2012interpolation}
$$\widehat{\sigma}^2 = \frac{Z_n^\T\Gamma_n(\alpha)^{-1}Z_n}{n}.$$
Let $\{\sigma_0^2,\alpha_0\}$ be the data generating parameters with observations $Z(x_1),\cdots,Z(x_n)$. For any misspecified $\alpha_1$, if $\widehat{\sigma}_{1,n}^2$ is the maximum likelihood estimation of $L_n(\sigma^2,\alpha_1)$, then $\widehat{\sigma}_{1,n}^2\alpha_1^{2\nu}\to \sigma_0^2\alpha_0^{2\nu}$  as $n\to\infty$ with probability $1$ under $P_0$ when $\cup_{n=1}^\infty\{x_n\}$ is bounded and infinite \citep{zhang2004inconsistent,kaufman2008covariance}.
Moreover, $\sqrt{n}\left(\frac{\widehat{\sigma}_{1,n}^2\alpha_1^{2\nu}}{\sigma_0^2\alpha_0^{2\nu}}-1\right)\to \mathcal{N}(0,2)$ as $n\to \infty$ \citep{du2009fixed,wang2011fixed,kaufman2013role}. As a result, even if we do not know the true parameters $\{\alpha_0,\sigma^2_0\}$, we can choose an arbitrary, possibly misspecified, decay parameter $\alpha_1$ and find the maximum likelihood estimate of the spatial variance $\widehat{\sigma}^2_{1,n}$. The resulting Gaussian measure is asymptotically equivalent to the Gaussian measure corresponding to the true parameter.

\subsection{Prediction and kriging}\label{subsec: prediction_kriging_euclidean}
Gaussian processes are widely deployed in spatial or nonparametric regression models to carry out model-based predictive inference. Given a new location $x_0$, the best linear unbiased predictor (BLUP) for $Z_0=Z(x_0)$ is given by
$$\widehat{Z}_n(\alpha)=\gamma_n(\alpha)^\T\Gamma_n(\alpha)^{-1}Z_n,$$
where $(\gamma_n(\alpha))_i = \frac{\left(\alpha\|x_0-x_i\|\right)^{\nu}}{\Gamma(\nu)2^{\nu-1}}K_\nu\left(\alpha\|x_0-x_i\|\right)$. Then 
$$\frac{\EE_{\sigma_0^2,\alpha_0}(\widehat{Z}_n(\alpha_1)-Z_0)^2}{\EE_{\sigma_0^2,\alpha_0}(\widehat{Z}_n(\alpha_0)-Z_0)^2}\xrightarrow{n\to\infty} 1,~~\frac{\EE_{\widehat{\sigma}_{1,n}^2,\alpha_1}(\widehat{Z}_n(\alpha_1)-Z_0)^2}{\EE_{\sigma_0^2,\alpha_0}(\widehat{Z}_n(\alpha_1)-Z_0)^2}\xrightarrow{n\to\infty} 1,$$
\changedreviewerone{where $\EE$ is the expectation with respect to the measure characterised by the parameter or spectral density (see Section~\ref{sec:MaternMfd}) in the subscript.}
As a result, any misspecified $\alpha$ still yields an asymptotically \changedreviewertwo{optimal BLUP} as long as $\sigma^2$ is replaced by its maximum likelihood estimate \citep{stein1993simple, kaufman2013role}. In the current manuscript, we develop parallel results for the $d$ dimensional compact Riemannian manifold $\MM$.

\section{Gaussian processes on compact Riemannian manifold}\label{sec:MaternMfd}
Henceforth, we assume that our domain of interest is a $d$-dimensional compact Riemannian manifold $\MM$ equipped with a Riemannian metric $g$. 
We denote the Laplace--Beltrami operator on $\MM$ by $-\Delta_g$ with eigenvalues $\lambda_n$ and eigenfunctions $f_n$, the volume form by $\mathrm{d}V_g$ and the volume of $\MM$ by $V_\MM$ \citep[see, e.g.,][for further details on operators and spectral theory on Riemannian manifolds]{kobayashi1963foundations,lee2018introduction,carmo1992riemannian}.

\subsection{Mat\'ern covariogram on compact Riemannian manifolds}\label{subsec: matern_compact_riemannian_manifold}

On a Riemannian manifold, where the linear structure of $\RR^d$ is missing, the standard definition of the Mat\'ern covariogram is no longer valid. A natural extension of the Mat\'ern covariogram to manifolds will consider replacing the Euclidean norm $\|x-y\|$ in (\ref{eq:materncovr}) by the geodesic distance $d(x,y)$. Unfortunately, this naive generalisation \changedreviewertwo{is not valid for $\nu=\infty$ \citep{feragen2015geodesic}, unless the manifold is flat}. If we restrict ourselves to spheres, \changedreviewertwo{Mat\'ern with $\nu\in(1/2,\infty)$ is still invalid \citep{gneiting2013strictly}}. Instead, some Mat\'ern-like covariograms including chordal, circular and Legendre Mat\'ern covariograms and other families of covariograms have been studied \citep{jeong2015covariance, porcu2016spatio, guinness2016isotropic, guella2018strictly,de2018regularity,alegria2021f}. However, these covariograms are constructed specifically with respect to the geometry of the sphere and do not generalise to generic compact Riemannian manifolds.   

\cite{whittle1963stochastic} showed that the Mat\'ern covariogram in Euclidean space admits a representation through a stochastic partial differential equation involving white noise and the Laplace operator $\Delta$. 
\cite{lindgren2011explicit} built on this stochastic partial differential equation approach to define the Mat\'ern covariogram on manifolds involving the \changedreviewertwo{Laplace--Beltrami} operator $\Delta_g$. This idea was further developed, both theoretically and practically, by several scholars \citep[see, e.g.,][among others]{bolin2011spatial,lang2015isotropic, herrmann2020multilevel,NEURIPS2020_92bf5e62, borovitskiy2021matern}. We state the definition of the Mat\'ern covariogram in the stochastic partial differential equation sense, which is a valid positive definite function for any $\nu$ on any compact Riemannian manifold $\MM$. 
\begin{definition}\label{def:mfd}
	Let $f_l$ be the \changedreviewertwo{orthonormal} eigenfunctions of $-\Delta_g$ and $\lambda_l\geq 0$ be the corresponding eigenvalues in ascending order. The Mat\'ern covariogram is defined by 
	$$k(x,y)=\frac{\sigma^2}{C_{\nu,\alpha}}\sum_{l=0}^\infty \left(\alpha^2+\lambda_l\right)^{-\nu-\frac{d}{2}}f_l(x)f_l(y),$$
	where $\displaystyle C_{\nu,\alpha}=\sum_{l=0}^\infty (\alpha^2+\lambda_l)^{-\nu-d/2}$ is a constant such that the average variance is $\sigma^2 = \frac{1}{V_\MM}\int_\MM k(x,x)\mathrm{dV_g}(x).$
	The corresponding spectral density is $$\rho(l) = \frac{\sigma^2}{C_{\nu,\alpha}}(\alpha^2+\lambda_l)^{-\nu-\frac{d}{2}}.$$
	Similarly, the squared exponential covariogram is
	$$k(x,y)=\frac{\sigma^2}{C_{\infty,\alpha}}\sum_{l=0}^\infty e^{-\frac{\lambda_l}{2\alpha^2}}f_l(x)f_l(y),$$
	where $\displaystyle C_{\infty,\alpha}=\sum_{l=0}^\infty e^{-\frac{1}{2\alpha^2}\lambda_l}$ is a constant such that the average variance is $\sigma^2 = \frac{1}{V_\MM}\int_\MM k(x,x)\mathrm{dV_g}(x).$
	The corresponding spectral density is 
	$$ \rho(n) = \frac{\sigma^2}{C_{\infty,\alpha}}e^{-\frac{\lambda_l}{2\alpha^2}}.$$
\end{definition}

\changedreviewertwo{
\begin{remark}
There are several commonly used parametric representations of the Mat\'ern covariogram. In particular, 
this article 
adopts %
the same 
parametric representation %
as the one in \cite{zhang2004inconsistent}, but different from \cite{borovitskiy2021matern}.
\end{remark}
}
If $\MM$ is a sphere, the covariograms defined above \changedreviewertwo{coincide with the Mat\'ern-like covariograms on spheres provided by \cite{guinness2016isotropic} and \cite{kirchner2020necessary}. As a result, we focus on a non-trivial generalisation to generic compact Riemannian manifolds}. The relation between the three parameters $(\alpha,\sigma^2, \nu)$ in the above definition and the coefficients in the stochastic partial differential equation representation is not straightforward \citep[see][for details]{lindgren2011explicit}. Note that for any $(\alpha,\sigma^2,\nu)$, the covariogram shares the same eigenbasis with the \changedreviewertwo{Laplace--Beltrami} operator $\Delta_g$. This property is not deemed restrictive for our ensuing development since we primarily focus on the Mat\'ern and squared exponential covariograms. Furthermore, this property offers crucial analytic tractability for several results developed subsequently. Hence, we refer to the Mat\'ern and squared exponential covariograms as in Definition~\ref{def:mfd} in the following sections.

\subsection{Identifiability}\label{subsec: identifiability}

In Euclidean domains, the integral test \citep{yadrenko1983spectral, stein2012interpolation} is a powerful tool to determine the equivalence of two Gaussian measures. However, such tests do not carry through to non-Euclidean domains as the spectrum on such manifolds is discrete. \cite{alegria2021f} studied the so called $\mathcal{F}-$family of covariograms on spheres and numerically deduced, without proof, the consistency of the maximum likelihood estimate of some parameters for this family. \cite{arafat2018equivalence} derived the equivalence of Gaussian measures on spheres and derived microergodic parameters of some covariograms excluding the Mat\'ern. All of the above results are built upon the Feldman--H\'ajek Theorem \citep{da2014stochastic}, which is valid for any metric space and, hence, applicable to compact Riemannian manifolds. Here, we generalise the above results to a Gaussian process with Mat\'ern and squared exponential covariograms on arbitrary compact Riemannian manifolds, also motivated by the Feldman--H\'ajek theorem. Therefore, we can still study the identifiability of these parameters by finding the microergodic parameters. 

\begin{lemma}\label{lem:test}
 Let $P_i$ ($i=1,2$) be mean zero Mat\'ern/squared exponential Gaussian random measures with spectral densities $\rho_i$. Then, $P_1\equiv P_2$ if and only if
	$$\sum_{l}\left|\frac{\rho_2(l)-\rho_1(l)}{\rho_1(l)}\right|^2<\infty.$$
\end{lemma}
\begin{proof}
	See Appendix \ref{app:lem:test}.
\end{proof}
From Definition~\ref{def:mfd}, $\rho_i$ is strictly positive so the denominator is always non-zero. The series test is a sufficient and necessary condition. This is a significant enhancement over the integral test in Euclidean spaces, which offers only a sufficient condition. Its importance to us will become clear after Theorem~\ref{thm:mfd}. Subsequently, we consider microergodic parameters of Gaussian processes on a manifold with the Mat\'ern covariogram. This is analogous to Theorem~2 in \cite{zhang2004inconsistent} for compact Riemannian manifolds.   

\begin{theorem}\label{thm:mfd}
	Let $P_i$, $i=1,2$, denote two Gaussian measures with the Mat\'ern covariogram parametrized by $\theta_i=\{\sigma^2_i,\alpha_i,\nu_i\}$. Then the following results hold.
	\begin{enumerate}[(A)]
		\item If $d\leq 3$, then  $P_{1}\equiv P_{2}$ if and only if $\sigma_1^2/C_{\nu_1,\alpha_1}=\sigma_2^2/C_{\nu_2,\alpha_2}$, $\nu_1=\nu_2$.
		\item If $d\geq 4$, then $P_{1}\equiv P_{2}$ if and only if $\sigma_1^2=\sigma_2^2$ and $\alpha_1=\alpha_2$, $\nu_1=\nu_2$. 
	\end{enumerate} 
\end{theorem}
\begin{proof}
	See Appendix \ref{app:thm:mfd}.
\end{proof}
Part~(A) of Theorem~\ref{thm:mfd} implies that if $d\leq 3$, then neither $\sigma^2$ nor $\alpha$ are \changedreviewertwo{identifiable or consistently estimable, while $\nu$ is identifiable. Part~(B) implies that when $d\geq 4$, all three parameters---$\sigma^2$,  $\alpha$ and $\nu$---are identifiable}. In Euclidean space, the smoothness parameter $\nu$ is typically assumed to be known and fixed when discussing fixed-domain asymptotic inference. \changedreviewertwo{In this specific Euclidean setting,} assuming $\nu_1=\nu_2=\nu$, (A) still holds while (B) holds for $d>4$; $d=4$ is still an unresolved problem in Euclidean space unless the domain is assumed to be bounded \citep{bolin2021equivalence}. This difference in behaviour between (A)~and~(B) can be attributed to the integral test being a sufficient condition in Euclidean spaces, which ensures only the equivalence of measures when $d\leq 3$; \citep[see][for details]{zhang2004inconsistent}. In $d>4$, \cite{A10} estimated the principal irregular term without the integral test and constructed consistent estimators for $\alpha$ and $\sigma^2$ directly. However, this construction does not hold for $d=4$. 

In contrast, the series test in Lemma~\ref{lem:test} is a \emph{sufficient and necessary} condition so that we can provide a condition for the equivalence of two measures with Mat\'ern covariograms over any dimension. The dimension also plays an important role in the manifold setting due to Weyl's Law \citep{li1987isaac, canzani2013analysis}. That is, the growth of the eigenvalues and their multiplicities are intertwined with the dimension $d$; further details are provided within the proof in Appendix~\ref{app:thm:mfd}. 
Another benefit of the sufficient and necessary condition is that the series test can be applied to the squared exponential covariogram, also known as the radial basis function, which can be viewed as a limiting case of the Mat\'ern covariogram when $\nu\to\infty$, \changedreviewertwo{as introduced in Definition \ref{def:mfd}}. 
Since the spectral density is not a polynomial, the integral test over Euclidean domains is invalid and the conditions for the equivalence of two squared exponential covariograms are intractable. In contrast, the following theorem resolves the equivalence of squared exponential covariograms on a compact manifold $\MM$.

\begin{theorem}\label{thm:rbf}
	Let $P_i$, for $i=1,2$, be Gaussian measures with squared exponential covariograms parametrised by $\theta_i = \{\sigma^2_i,\alpha_i\}$. Then $P_{1}\equiv P_{2}$ if and only if $\sigma_1^2=\sigma_2^2$ and $\alpha_1 = \alpha_2$.
\end{theorem}
\begin{proof}
	See Appendix \ref{app:thm:rbf}.
\end{proof}
Theorem~\ref{thm:rbf} shows that it is possible to have consistent estimators for both $\sigma^2$ and $\alpha$. So far we have developed formal results on the identifiability of parameters in the covariogram on a compact Riemannian manifold. Inference for identifiable parameters will proceed in customary fashion so we turn our attention to non-identifiable settings, i.e., the Mat\'ern covariogram with known $\nu$ on manifolds with dimension $d\leq 3$.

\subsection{Consistency of maximum likelihood estimation}\label{subsec: consistency_mle}
Since $\MM$ is compact, there is no increasing-domain asymptotic framework and $\cup_{n=1}^\infty\{x_n\}$ is always bounded. In the remaining sections, we assume that $\cup_{n=1}^\infty\{x_n\}$ is infinite, which is the standard assumption also known as the increasing sequence assumption \citep[also see][]{stein2012interpolation,zhang2004inconsistent,kaufman2013role}. Let $\{\sigma_0,\alpha_0\}$ be the data generating parameter (oracle) and let $\widehat{\sigma}_{1,n}^2$ be the maximum likelihood estimate of $\sigma^2$ obtained by maximising $L_n(\sigma^2,\alpha_1)$ with a misspecified $\alpha_1$. The following theorem is analogous to Theorem~3 in \cite{zhang2004inconsistent} for compact Riemannian manifolds.

\begin{theorem}\label{thm:consist}
	Under the setting of Theorem~\ref{thm:mfd}, assuming $\cup_{n=1}^\infty\{x_n\}$ is infinite, we obtain
	$$\frac{\widehat{\sigma}_{1,n}^2}{C_{\nu,\alpha_1}}\xrightarrow{n\to\infty} \frac{\sigma_0^2}{C_{\nu,\alpha_0}},~P_0~ a.s. $$
\end{theorem}
\begin{proof}
	See Appendix \ref{app:thm:consist}.
\end{proof}
In Euclidean space, $\widehat{\sigma}_{1,n}^2/C_{\nu,\alpha_1}$ is asymptotically Gaussian. We conjecture that this asymptotic normality still holds on Riemannian manifolds. However, this result relies on specific constructions in Euclidean space \citep{daqing2010fixed}, which become invalid for manifolds. A formal proof is beyond the scope of the current manuscript and we intend to pursue this development in future investigations. In Section~\ref{sec:sphere} we present a numerical simulation experiment to demonstrate the asymptotic (normal) behaviour of this parameter on spheres.

\subsection{Prediction}\label{subsec: prediction}
Given a new location $x_0\in \MM\backslash\{x_i\}_{i=1}^n$, the best linear unbiased predictor for $Z_0=Z(x_0)$ \changedreviewerone{under a covariance function $k_\rho$ characterised by its spectral density $\rho$} is given by
$$\widehat{Z}_n(\rho)=\gamma_n(\rho)^\T\Gamma_n(\rho)^{-1}Z_n,$$
\changedreviewerone{where $\gamma_n(\rho) = \frac{1}{\sigma^2}k_\rho(x_0,x_i)$ and $\{\Gamma_n(\rho)\}_{ij}=\frac{1}{\sigma^2}k_\rho(x_i,x_j)$.}

\cite{kirchner2020necessary} and \cite{bolin2021equivalence} generalise the results of asymptotic \changedreviewertwo{optimality of the BLUP based on a misspecified scale parameter in Euclidean spaces \citep{stein1993simple} to metric spaces. That is, the prediction error of the BLUP under a misspecified scale parameter is asymptotically the same as the error of the BLUP under the true parameter.} 
If the domain is a compact Riemannian manifold and the covariograms are Mat\'ern, then two covariance operators share the same eigenbasis; this is the setting described in Section 5.1 of \cite{kirchner2020necessary} as a special case of Theorem~3.1 therein.  We rephrase it in the following lemma with some modifications to fit the Mat\'ern covariograms on a compact Riemannian manifold with a different and simpler proof.
\begin{lemma}\label{lem:prediction}
	Let $\rho_0,\rho_1$ be the spectral densities of two Gaussian measures on $\MM$ with Mat\'ern covariograms. Given $x_0\in \MM\backslash\{x_i\}_{i=1}^n$, let $\widehat{Z}_n(\rho_i)$ be the best linear unbiased predictor of $Z_0\coloneqq Z(x_0)$ based on observations $\{Z(x_1), \cdots, Z(x_n)\}$ with $\{x_i\}_{i=1}^\infty$ being infinite \changedreviewertwo{and having $x_0$ as an accumulation point}, where $\rho_i$ is the spectral density of $Z(\cdot)$. If there exists a real number $c$ such that $\displaystyle \lim_{m\to \infty} \frac{\rho_1(m)}{\rho_0(m)}=c$, then: \begin{enumerate}[(i)]
		\item $\displaystyle \frac{\EE_{\rho_0}(\widehat{Z}_n(\rho_1)-Z_0)^2}{\EE_{\rho_0}(\widehat{Z}_n(\rho_0)-Z_0)^2}\xrightarrow{n\to\infty }1,$
		\item $\displaystyle \frac{\EE_{\rho_1}(\widehat{Z}_n(\rho_1)-Z_0)^2}{\EE_{\rho_0}(\widehat{Z}_n(\rho_1)-Z_0)^2}\xrightarrow{n\to\infty }c$. 
	\end{enumerate}
\end{lemma}
\begin{proof}
	See Appendix \ref{app:lem:prediction}.
\end{proof}
Focusing on the parameters in a Mat\'ern covariogram, let $\widehat{\sigma}_{1,n}^2$ be the maximum likelihood estimate of $L_n(\sigma^2,\alpha_1)$ and $\rho_i$ be the spectral density of the Mat\'ern covariogram with decay parameter $\alpha_i$.
\begin{theorem}\label{thm:prediction}
	Under the same conditions as in Theorem \ref{thm:mfd} and Lemma \ref{lem:prediction}, let $\sigma_1^2 = \sigma_0^2C_{\nu,\alpha_1}/C_{\nu,\alpha_0}$, then
	$$\frac{\EE_{\sigma_0^2,\alpha_0}(\widehat{Z}_n(\alpha_1)-Z_0)^2}{\EE_{\sigma_0^2,\alpha_0}(\widehat{Z}_n(\alpha_0)-Z_0)^2}\xrightarrow{n\to\infty} 1,~~
	\frac{\EE_{\widehat{\sigma}_{1,n}^2,\alpha_1}(\widehat{Z}_n(\alpha_1)-Z_0)^2}{\EE_{\sigma_0^2,\alpha_0}(\widehat{Z}_n(\alpha_1)-Z_0)^2}\xrightarrow[P_0 ~a.s.]{n\to\infty} 1.$$
\end{theorem}
\begin{proof}
	See Appendix \ref{app:thm:prediction}.
\end{proof}
Note that Lemma \ref{lem:prediction} and Theorem~\ref{thm:prediction} offer the manifold versions of Theorems~3~and~4 in \cite{kaufman2013role}. 

\section{Mat\'ern on spheres}\label{sec:sphere}

We now consider Gaussian processes with the Mat\'ern covariogram on the $d$-dimensional sphere $S^d$, including two popular manifolds in spatial statistics: the circle $S^1$ and sphere $S^2$. We show that all theorems in the previous sections hold for $S^d$ with $d=1,2,3$. As earlier, we assume that $P_i, ~i = 1,2$, are two Gaussian measures on $S^d$ with Mat\'ern covariogram parameters $\{ \sigma_i^2,\alpha_i, \nu\}$. 

\begin{theorem}\label{thm:spheres}
	For spheres with dimension $d=1,2,3$, the following results are true:
	\begin{enumerate}
		\item $P_1\equiv P_2$ if and only if $\sigma_1^2/C_{\nu,\alpha_1}=\sigma_2^2/C_{\nu,\alpha_2}$, so neither $\sigma^2$ nor $\alpha$ can be consistently estimated. 
		\item Let the data generating parameters be  $\{\sigma_0,\alpha_0\}$ and $\widehat{\sigma}_{1,n}^2$ be the maximum likelihood estimation of $L_n(\sigma^2,\alpha_1)$ with misspecified $\alpha_1$ based on increasing sequence $\{x_i\}_{i=1}^n$. Then, $$\frac{\widehat{\sigma}_{1,n}^2}{C_{\nu,\alpha_1}}\xrightarrow{n\to\infty} \frac{\sigma_0^2}{C_{\nu,\alpha_0}},~P_0~a.s.$$
		\item Given $x_0\in \MM\backslash\{x_i\}_{i=1}^n$, let $\widehat{Z}_n$ be the best linear unbiased predictor of $Z_0\coloneqq Z(x_0)$ based on observations $\{Z(x_1), \cdots, Z(x_n)\}$ with $\{x_i\}_{i=1}^\infty$ being infinite, then
		$$\frac{\EE_{\widehat{\sigma}_{1,n}^2,\alpha_1}(\widehat{Z}_n(\alpha_1)-Z_0)^2}{\EE_{\sigma_0^2,\alpha_0}(\widehat{Z}_n(\alpha_1)-Z_0)^2}\xrightarrow{n\to\infty} 1,~P_0~a.s.$$
	\end{enumerate}
\end{theorem}
\begin{proof}
	See Appendix \ref{app:thm:spheres}.
\end{proof}
Next, we consider two concrete examples: the circle $S^1$ and the sphere $S^2$.

\subsection{Mat\'ern covariogram on circle}
First, we recall the simplified form of the Mat\'ern covariogram on $S^1$ \citep{NEURIPS2020_92bf5e62}:
\begin{lemma}\label{lem:circle}
	When $\MM=S^1\subset \RR^2$ and $\nu=1/2+s$, $s\in\NN$, the Mat\'erm covariogram is given by
	\begin{equation}\label{eqn:MaternCovS1}
		k(x,y) = \frac{\sigma^2}{C'_{\nu,\alpha}}\sum_{k=0}^s a_{s,k}(\alpha(|x-y|-1/2))^k\mathrm{hyp}^k(\alpha(|x-y|-1/2)),~~x,y\in S^1,
	\end{equation}
	where $C'_{\nu,\alpha}$ is chosen so that $k(x,x) = \sigma^2$, $\mathrm{hyp}^k$ is $\cosh$ when $k$ is even and $\sinh$ when $k$ is odd, $a_{s,k}$ are constants depending on $\nu$ and $\alpha$; see \cite{NEURIPS2020_92bf5e62} for details.
\end{lemma} 

Note that $x-y\coloneqq \theta_x-\theta_y~\mod{1}$ for $x=e^{2\pi i\theta_x}$ and $y=e^{2\pi i\theta_y}$. Therefore, the Mat\'ern covariogram is ``stationary'' with respect to this group addition instead of the standard addition in Euclidean space. The corresponding spectral density is given by 
\begin{equation}\label{eqn:MaternSpeS1}
	\rho(n)=\frac{2\sigma^2\alpha\sinh(\alpha/2)}{C'_{\nu,\alpha}(2\pi)^{1-2\nu}}\left(\alpha^2+4\pi^2n^2\right)^{-\nu-1/2},~~n\in\ZZ.
\end{equation}

In particular, when $\nu=1/2$, the covariogram and spectral densities admit simple forms:
$$k(x,y)=\frac{\sigma^2}{\cosh(\alpha/2)}\cosh\left(\alpha(|x-y|-1/2)\right),$$
$$\rho(n)=2\sigma^2\alpha\tanh(\alpha/2)(\alpha^2+4\pi^2n^2)^{-1}.$$
Figure~1(a) depicts a covariogram with $\nu=1/2$, $\alpha=2$, and $\sigma^2=1$. Note that $|x-y|=1/2$ means that $x$ and $y$ are antipodal points so the correlation attains a minimum. Figure~1(b) shows a set of simulated $Z$'s with different values of $\alpha$. It is clear that the smaller values of $\alpha$ generate smoother random fields as the correlation grows larger. 

\begin{figure}
	\centering
	\subfloat[]{\includegraphics[width = 0.37\linewidth]{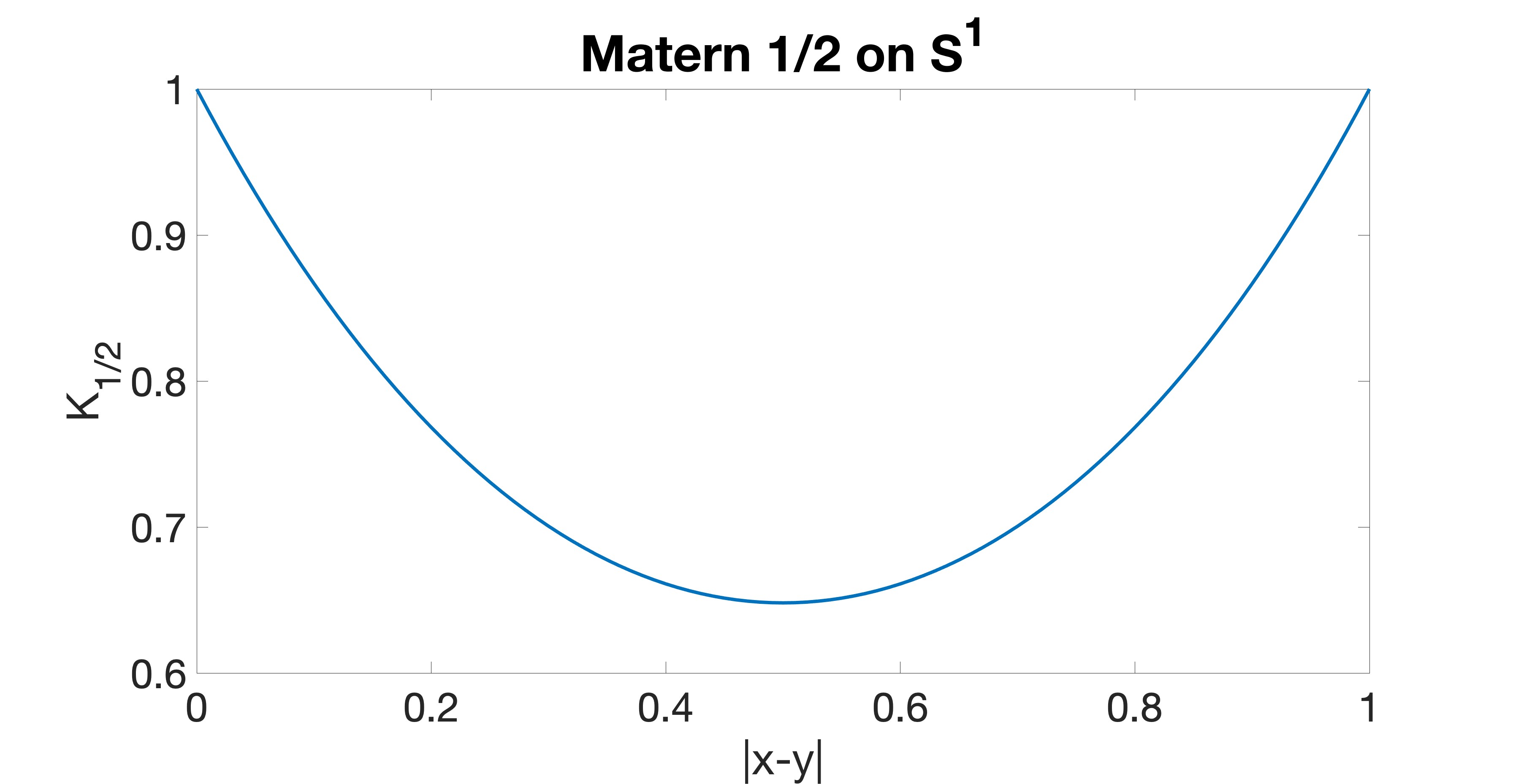}}	\label{fig:Mat1/2_S1}
	\subfloat[]{\includegraphics[width = 0.62\linewidth]{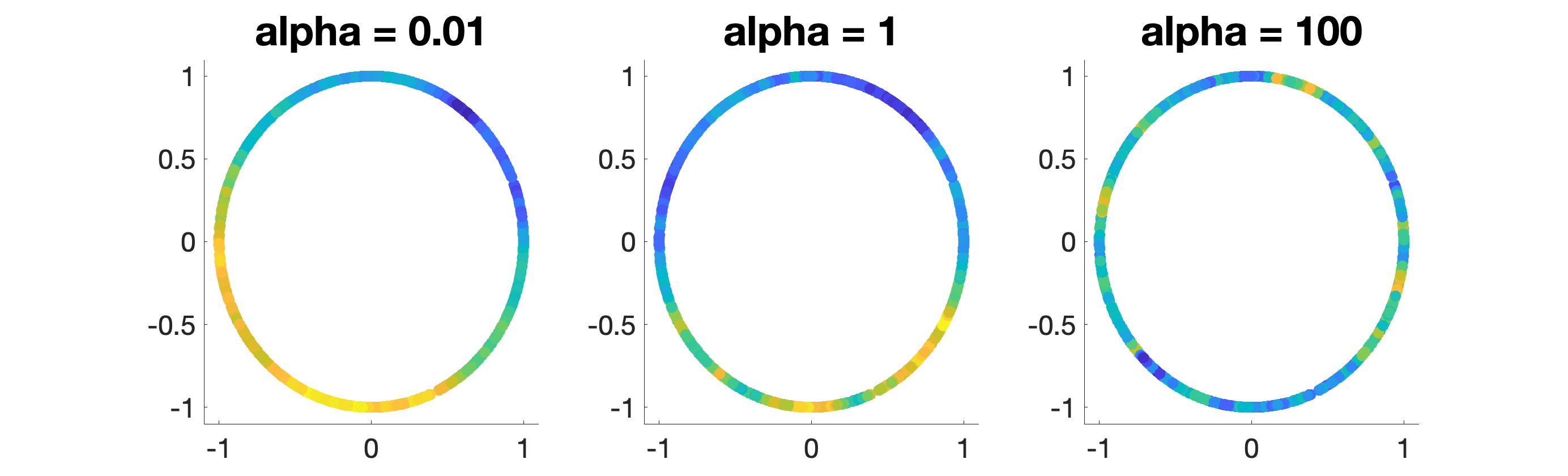}}	\label{fig:S1_simulation}
	\caption{(a) Covariogram of Mat\'ern 1/2 on $S^1$; (b): Sample fields with $\sigma^2=0.1$, $\nu=1/2$, $\alpha\in\{0.01,1,100\}$.} 
\end{figure}

\begin{corollary}\label{cly:circle1/2}
	Let $\nu=1/2$, then $P_{1}\equiv P_{2}$ if and only if $\sigma_1^2\alpha_1 \tanh(\alpha_1/2)=\sigma_2^2\alpha_2 \tanh(\alpha_2/2)$, so neither $\sigma^2$ nor $\alpha$ can be consistently estimated. 
\end{corollary}

For a general $\nu=1/2+s,~s\in\NN$, the normalising constant is
$$C'_{\nu,\alpha}=\sum_{k=0}^s a_{s,k}(-\alpha/2)^k\mathrm{hyp}^k(-\alpha/2).$$
\changedreviewertwo{We point out that this $C'_{\nu,\alpha}$ is different from the $C'_{\nu,\alpha}$ in Definition \ref{def:mfd} when $\MM = S^1$.} Although we cannot express $C'_{\nu,\alpha}$ as an elementary function, we can still find the microergodic parameter for any $\nu=s+1/2,~s\in\ZZ$:
\begin{corollary}\label{cly:circle}
	Let $\nu=1/2+s,~s\in\ZZ$, then $P_{1}\equiv P_{2}$ if and only if $\sigma_1^2\alpha_1 \sinh(\alpha_1/2)/C'_{\nu,\alpha_1}=\sigma_2^2\alpha_2 \sinh(\alpha_2/2)/C'_{\nu,\alpha_2}$, so neither $\sigma^2$ nor $\alpha$ can be consistently estimated. 
\end{corollary}

Figure~2 shows that \changedreviewertwo{$\widehat{\sigma}^2_{1,n}\to \sigma_1^2\coloneqq \frac{\sigma_0^2\alpha_0\sinh(\alpha_0/2)}{C'_{\nu,\alpha_0}}\frac{C'_{\nu,\alpha_1}}{\alpha_1\sinh(\alpha_1/2)}$ as shown by the horizontal line} and the empirical distribution of $\sqrt{n}\left(\frac{\widehat{\sigma}^2_{1,n}}{\sigma_1^2}-1\right)$ is $N(0,2)$, for $\nu=1/2$, $\sigma_0=0.1$, $\alpha_0=2\neq \alpha_1=1$. Panel~(a) supports Theorem~\ref{thm:spheres} empirically. That is, although $(\sigma^2,\alpha,\nu)$ are not consistently estimable, the microergodic parameter \changedreviewertwo{$\frac{\sigma^2\alpha\sinh(\alpha/2)}{C'_{\nu,\alpha}}$} is consistently estimable. Panel~(b) supports our conjecture after Theorem \ref{thm:consist} empirically.

\begin{figure}
	\centering
	\subfloat[]{\includegraphics[width = 0.49\linewidth]{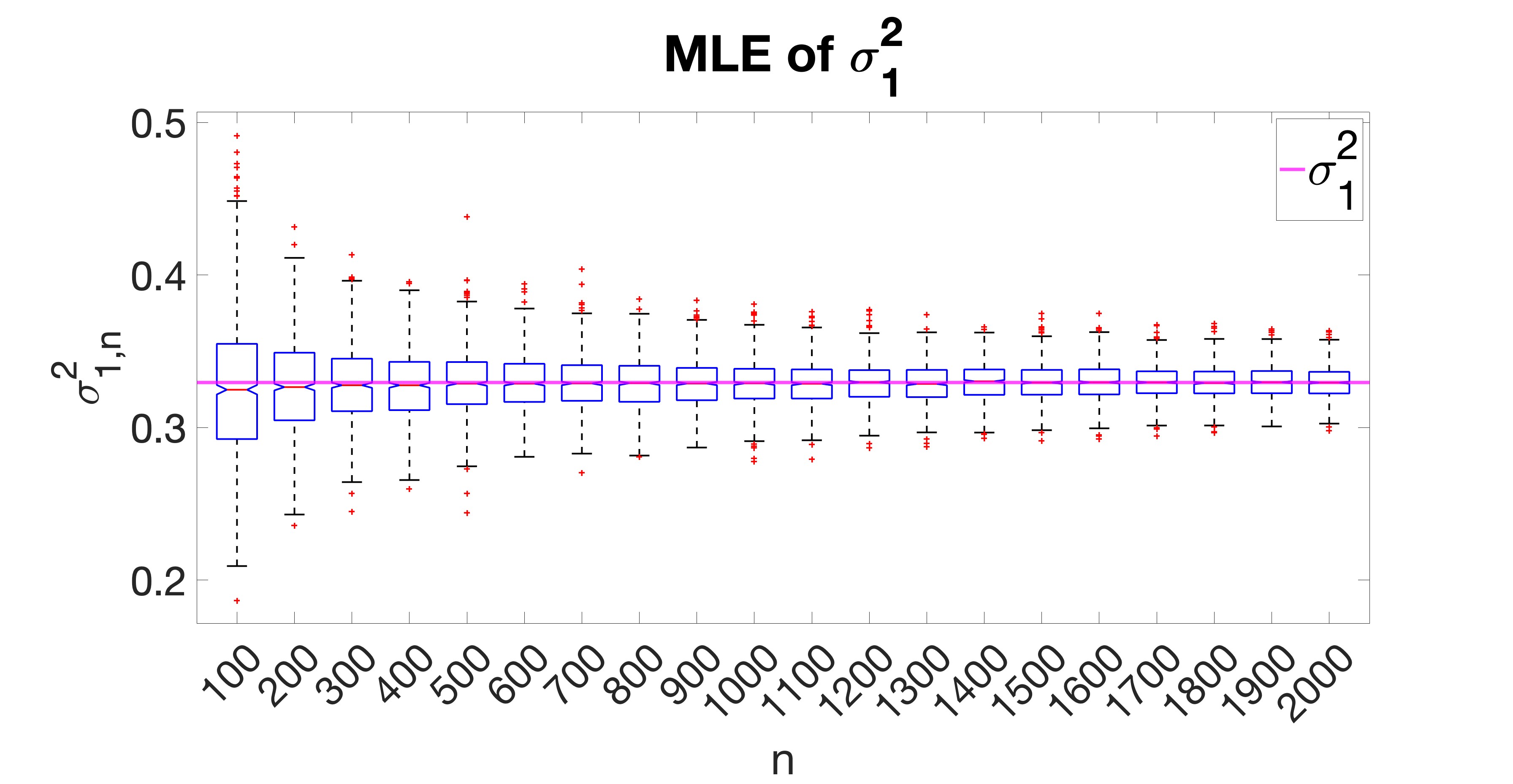}}\label{fig:S1_MLE}
	\subfloat[]{\includegraphics[width = 0.49\linewidth]{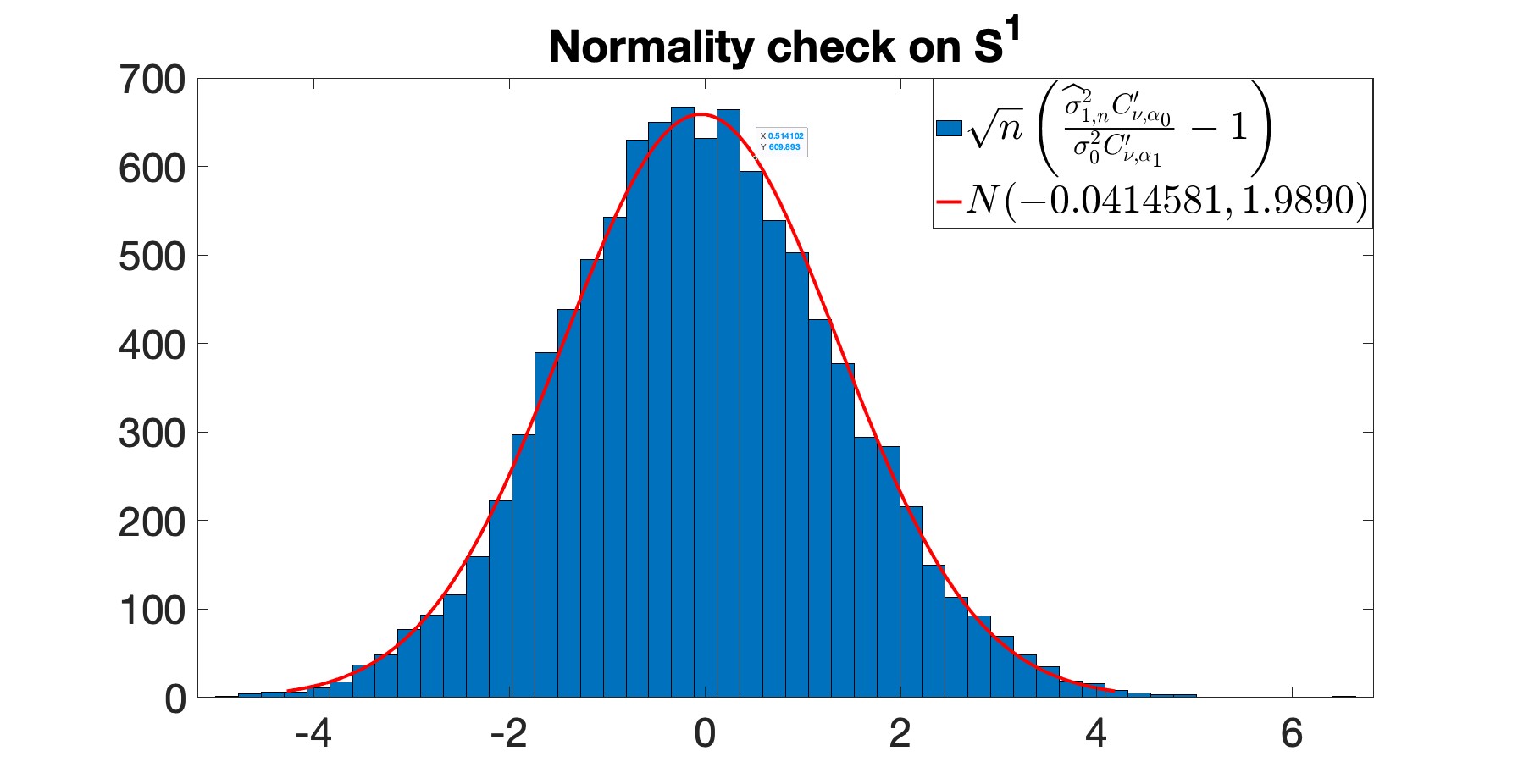}}\label{fig:S1_norm}
	\caption{(a) $\widehat{\sigma}^2_{1,n}$ v.s. $\sigma_1^2$; (b): Distribution of $\sqrt{n}\left(\frac{\widehat{\sigma}_{1,n}^2}{\sigma_1^2}-1\right)$.} 
\end{figure}

\subsection{Mat\'ern covariogram on the sphere}
On a sphere $S^2$, the Mate\'rn covariogram is more complicated \citep{NEURIPS2020_92bf5e62}: 
\begin{lemma}\label{lem:sphere}
	The Mat\'ern covariogram on $\MM=S^2$ with $\nu > 0$ is
	\[
	k(x,y)=\frac{\sigma^2}{C_{\nu,\alpha}}\sum_{l=0}^\infty\left(\alpha^2+l(l+1)\right)^{-\nu-1}c_{l}\mathcal{L}_l(\cos(d_M(x,y)))
	\]
	and its spectral density is given by
	\[
	\rho(l)=\frac{\sigma^2}{C_{\nu,\alpha}} \left(\alpha^2+l(l+1)\right)^{-\nu-1}\;, 
	\]
	where $d_M(\cdot,\cdot)$ is the geodesic distance on $S^2$, $\mathcal{L}_l$ is the Legendre polynomial of degree $l$:
	\[
	\mathcal{L}_l(z)=\sum_{k=0}^{\floor*{l/2}}(-1)^k\frac{l!(l-k-\frac{1}{2})!}{k!(l-2k)!}(2z)^{n-2k}\;,\quad \mbox{and} 
	\]
	\[
	c_{l}=\frac{(2l+1)\Gamma(3/2)}{2\pi^{3/2}},~~C_{\nu,\alpha} = \frac{\Gamma(3/2)}{{8\pi^{5/2}}}\sum_{l=0}^\infty (2l+1)\left(2\nu\alpha^2+l(l+1)\right)^{-\nu-1}
	\]
\end{lemma}
\changedreviewertwo{
\begin{remark}
The index $l$ in the above covariance function is different from the index $l$ in Definition \ref{def:mfd}. In fact, each Legendre polynomial corresponds to multiple spherical harmonics, so the spectral density does not contain the $c_l$ constants anymore. 
\end{remark}
}
Unlike Lemma~\ref{lem:circle}, where $\nu$ is required to be a half-integer, here $\nu$ can be any positive number. However, the covariogram now involves an infinite series, which needs to be approximated when $x\neq y$. Approximating a function on $S^2$ is known as the ``scatter data interpolation problem'' \citep{narcowich1998stability} and preserving the positive definiteness is known as the stability problem \citep{
kunis2009note}. For the Mat\'ern covargioram considered in this manuscript, we adopt a natural and simple approximation using the partial sum of an infinite series. The following theorem controls the approximation error and ensures the positive definiteness of the approximated covariogram.

\begin{theorem}\label{thm:cov2partial}
	For the partial sum
	$$k^L(x,y)=\frac{\sigma^2}{C_{\nu,\alpha}}\sum_{l=0}^L\left(\alpha^2+l(l+1)\right)^{-\nu-1}c_{l}\mathcal{L}_l(\cos(d_M(x,y))),$$
	the approximation error is controlled by 
	$$|k^L(x,y)-k(x,y)|\leq \epsilon\coloneqq \changedreviewertwo{\frac{12\pi\sigma^2}{\sum_l (2l+1)(\alpha^2+l(l+1))^{-\nu-1}} L^{-2\nu}}.$$
	Given observations $x_1,\cdots,x_n$ with minimal separation $q=\inf_{i\neq j}d(x_i,x_j)$, the approximated covariance matrix $\{k^L(x_i,x_j)\}_{ij}$ is positive definite for any 
	$$L>\changedreviewertwo{\left(\frac{12\pi n\sigma^2}{\xi_\rho(q)\sum_l (2l+1)(\alpha^2+l(l+1))^{-\nu-1}}\right)^{\frac{1}{2\nu}}},$$
	where $\xi_\rho(q)$ is a constant depending on the spectral density $\rho$ and minimal separation $q$; see the proof for more details. 
\end{theorem}
\begin{proof}
	See Appendix \ref{app:thm:cov2partial}.
\end{proof}
The above result implies that the computational cost is of order $\epsilon^{-\frac{1}{2\nu}}$ as $\epsilon\to 0$. Larger values of $\nu$ imply smoother random fields that require smaller values of $N$ to approximate the covariogram. In practice, we can first calculate $\xi_\rho(q)$, which is computationally practicable because of the closed-form representation (see Appendix \ref{app:thm:cov2partial} for details), and then choose $N$.

Figure~3(a) presents the covariogram with $\nu=1/2$, $\alpha=1$, and $\sigma^2=1$. Note that $d(x,y)=\pi$ means that $x$ and $y$ are antipodal points so the correlation reaches the minimum. Figure~3(b) shows some simulated $Z$'s with different $\alpha$'s. Similar to $\MM = S^1$, smaller values of $\alpha$ lead to smoother random fields. 


\begin{figure}
	\centering
	\subfloat[]{\includegraphics[width = 0.44\linewidth]{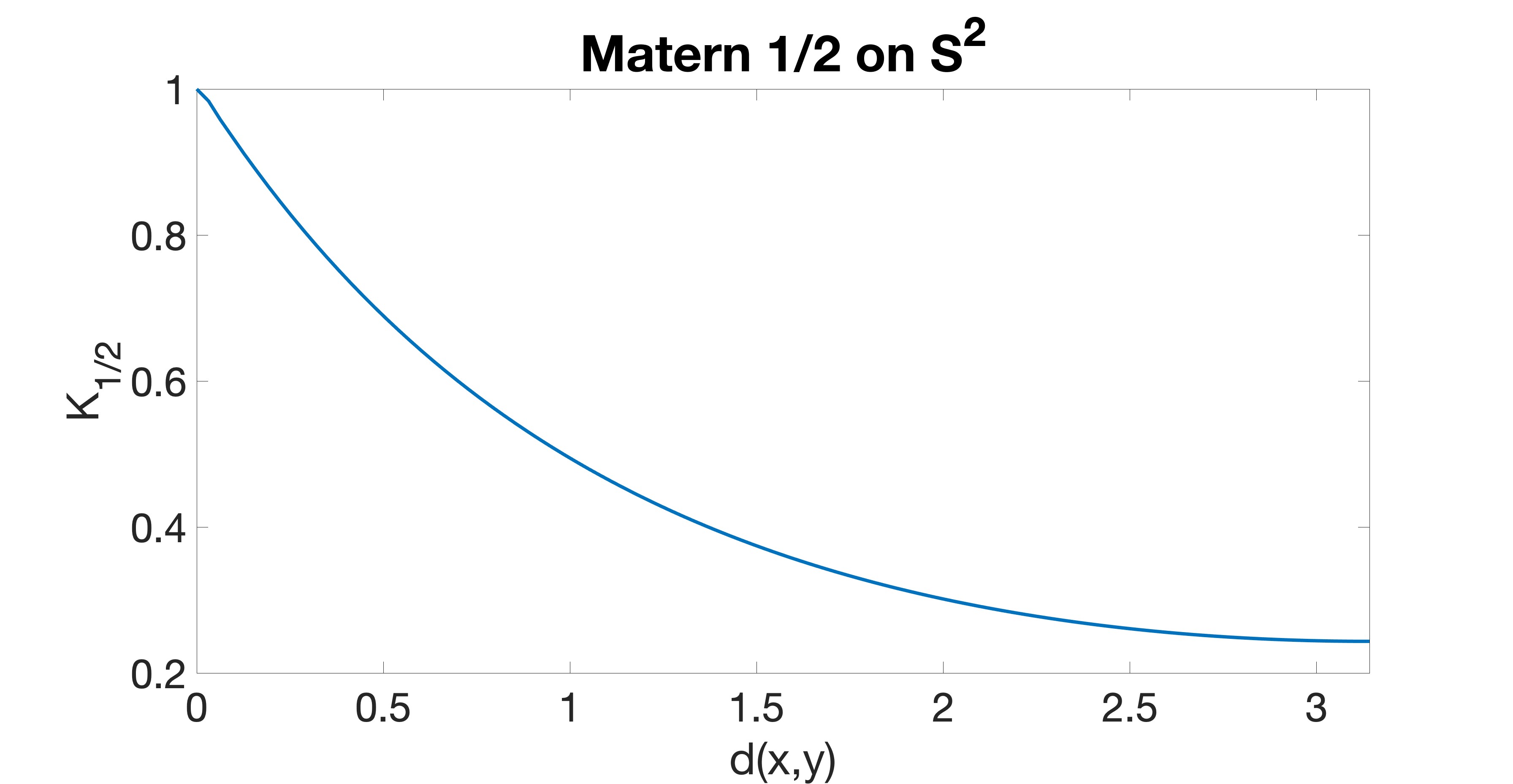}}\label{fig:Mat1/2_S2}
	\subfloat[]{\includegraphics[width = .55\linewidth]{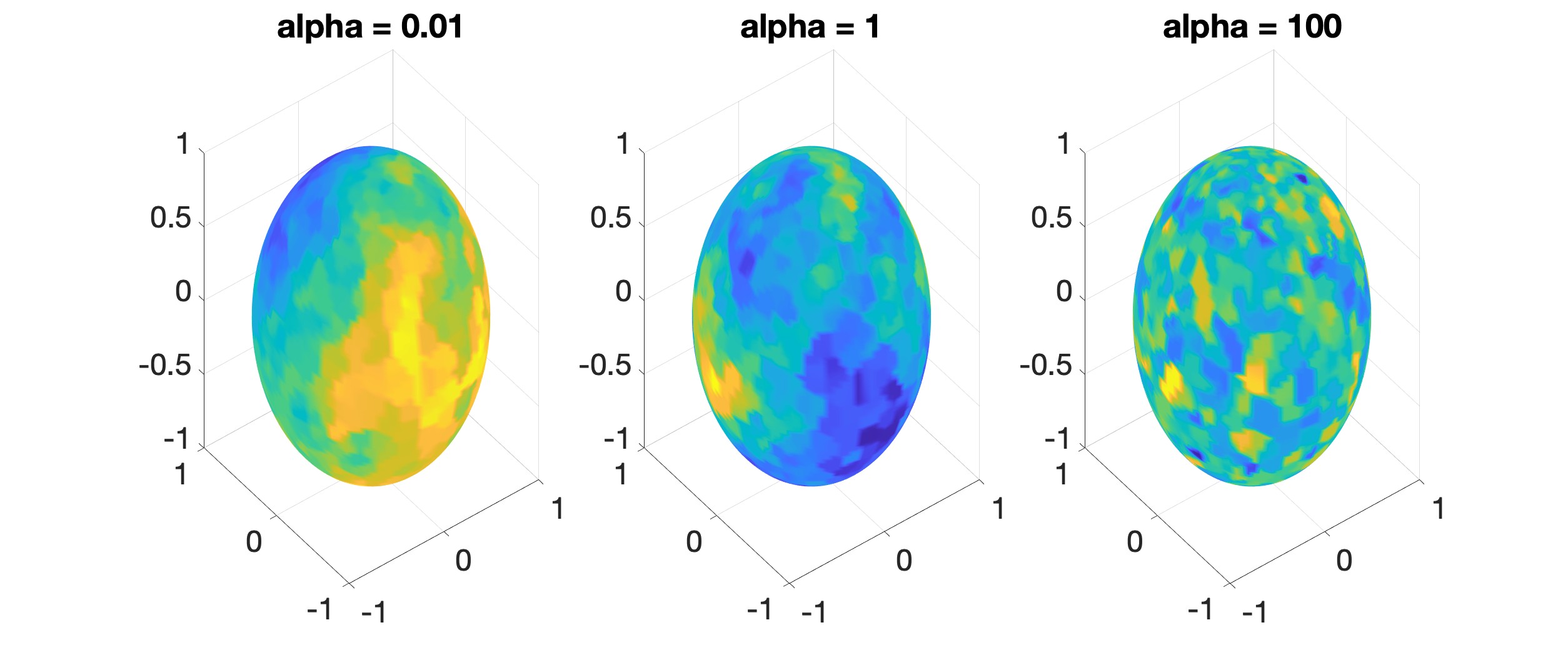}}\label{fig:S2_simulation}
	\caption{(a) Covariogram of Mat\'ern 1/2 on $S^2$; (b): Sample fields with $\sigma^2=0.1$, $\nu=1/2$, $\alpha\in\{0.01,1,100\}$.} 
\end{figure}

However, due to the bias introduced by the partial sum, we do not have access to the ground truth covariogram, so the analogue of Figure~2 is not available anymore. Similar issues arise in approximations to the Mat\'ern on a compact manifold \citep{sanz2020spde}. Instead, we show the theoretical results on microergodic parameters analogous to Corollary~\ref{cly:circle}: 
\begin{corollary}\label{cly:sphere}
	$P_{\theta_1}\equiv P_{\theta_2}$ if and only if $\sigma_1^2/C_{\nu,\alpha_1}=\sigma_2^2/C_{\nu,\alpha_2}$, so neither $\sigma^2$ nor $\alpha$ can be consistently estimated. 
\end{corollary}




\section{Discussion}\label{sec: discussion}
This article has formally developed some theoretical results on statistical inference for Gaussian processes with Mat\'ern covariograms on compact Riemannian manifolds. Our focus has primarily been on the identifiability and consistency (or lack thereof) of the covariogram parameters and of spatial predictions. For the Mat\'ern and squared exponential covariograms, we provide a sufficient and necessary condition for the equivalence of two Gaussian random measures through a series test and derive identifiable and consistently estimable microergodic parameters for an arbitrary dimension $d$. Specifically for $d\leq 3$, we formally establish the consistency of maximum likelihood estimates of the parameters and the asymptotic normality of the best linear unbiased predictor under a misspecified decay parameter. The circle and sphere are analysed as two examples with corroborative numerical experiments. 

We anticipate that the results developed here will generate substantial future work in this domain. For example, as we have alluded to earlier in the article, in Euclidean spaces we know that the maximum likelihood estimate of $\sigma^2$ is asymptotically normal: $\sqrt{n}\left(\frac{\widehat{\sigma}_{1,n}^2}{C_{\nu,\alpha_1}}-\frac{\sigma_0^2}{C_{\nu,\alpha_0}}\right)\to N(0,2)$. While our numerical experiments lead us to conjecture that an analogous result holds for compact Riemannian manifolds, a formal proof may well require substantial new machinery that we intend to explore further. Next, we conjecture that two measures with the Mat\'ern covariogram are equivalent on $\RR^4$ if and only if they have the same decay and spatial variance parameters. We know this result holds for manifolds with $d=4$, but a formal proof for 
$\RR^4$ has not yet been established. Based upon similar reasonings we conjecture that two measures with squared exponential covariograms are equivalent on Euclidean spaces if and only if they have the same decay and spatial variance parameters. 

Another future generalisation is to consider covariograms on compact Riemannian manifolds that are not simultaneously diagonalisable, whose asymptotically optimal linear predictor has been studied in \cite{kirchner2020necessary}. Nevertheless, issues pertaining to the equivalence of measures, derivation of microergodic parameters and consistency of maximum likelihood estimates remain unresolved. Furthermore, covariograms that offer scientific interpretation in practical inference need to be explored. In this regard, it is worth remarking that although our results are primarily concerned with maximum likelihood estimates, they will provide useful insights into Bayesian learning on manifolds. For example, the failure to consistently estimate certain (non-microergodic) parameters will inform Bayesian modellers that inference for such parameters will always be sensitive to their prior specifications. This will open up new avenues of research in specifying prior distributions for microergodic parameters. Formal investigations into the consistency of the posterior distributions of Mat\'ern covariogram parameters on manifolds are of inferential interest and may benefit from some of our developments in the current manuscript. 

Other avenues for future developments will relate to computational efficiency of Gaussian processes on manifolds. Here, a natural candidate for explorations is the tapered covariogram on manifold to introduce sparsity in the covariance matrix \citep{furrer2006covariance}. Since our domain in the current manuscript is compact, unlike in Euclidean domains, further compact truncation is redundant. One can explore the development of new ``tapered'' covariograms that achieve positive-definiteness and sparsity. Other approaches that induce dimension reduction based on conditional expectations, such as Gaussian predictive processes \citep{banerjee2008gpp}, may be explored on compact Riemannian manifolds since these low-dimensional processes are induced by any valid probability measure, although the choice of inputs to define the lower dimensional subspace will need to be addressed. On the other hand, sparse processes resulting from approximations using directed acyclic graphs \citep{datta16} are less natural for modelling data on manifolds since they depend on well-defined neighbours of inputs, which are less obvious to define outside of Euclidean spaces. Nevertheless, \cite{datta16b} developed adaptive Nearest-Neighbour Gaussian processes for massive space-time data sets on Euclidean spaces that selected neighbours using the covariance kernel as a metric for proximity. Such an approach holds promise in modelling massive data sets on manifolds.   

In addition, asymptotic properties of estimates under tapering are of interest and have, hitherto, been explored only in Euclidean domains \citep{kaufman2008covariance, du2009fixed} and without the presence of measurement error processes (``nuggets''). Inference for Gaussian process models with measurement errors (nuggets) on compact manifolds also present novel challenges and can constitute future work. Identifiability and consistency of the nugget in Euclidean spaces have only recently started receiving attention \citep{tang2019identifiability}. However, the developments for Euclidean spaces do not easily apply to compact Riemannian manifolds; hence new tools will need to be developed. On complex or unknown domains, the eigenvalues and eigenfunctions of the Laplacian operator need to be estimated \citep{belkin2007convergence}. Asymptotic analysis of estimation in the spectral domain should be closely related to the frequency domain. 
Finally, since compact manifolds are distinct from non-compact manifolds, both geometrically and topologically, generalisation to non-compact Riemannian manifolds is of interest, where the spectrum is not discrete. Analytic tools on non-compact manifolds will need to be developed. 

\acks{DL would like to thank Viacheslav Borovitskiy, Yidan Xu and Aritra Halder for helpful discussions. DL was supported by NIH/NCATS award UL1 TR002489, NIH/NHLBI award R01 HL149683 and NIH/NIEHS award P30 ES010126. DL and SB were supported by NSF awards DMS-1916349, IIS-1562303, and NIH/NIEHS award R01ES027027. WT acknowledges support from NSF awards DMS-2113779 and DMS-2206038, and from a startup grant at Columbia University.}



\appendix
\section{Proof of Lemma \ref{lem:test}}\label{app:lem:test}

Before proving Lemma \ref{lem:test}, we recall the following lemma \citep[Proposition B, Chapter III][]{yadrenko1983spectral}, also known as the Feldman--H\'ajek theorem:
\begin{lemma}\label{lem:B}
	$P_1\equiv P_2$ if and only if 
	\begin{enumerate}
		\item Operator $D=B_1^{-1/2}B_2B_1^{-1/2}-\mathrm{I}$ is Hilbert--Schmidt;
		\item Eigenvalues of $D$ are strictly greater than $-1$,
	\end{enumerate}
	where $B_i$ is the correlation operator of $P_i$ defined by:
	$$(B_ih)(x)\coloneqq \int_\MM \sum_{l=0}^\infty\rho_i(l) f_l(x)f_l(y){h(y)}\mathrm{dV_g}(y),~h\in L^2(\MM).$$
\end{lemma}

\begin{proof}{\bf{of Lemma \ref{lem:test}}.}
	By Lemma \ref{lem:B}, it suffices to check conditions 1 and 2. Let $\gamma^i_n$ be the eigenvalue of $B_i$ and $d_n$ be the eigenvalue of $D$. Observe that $f_n$ is an eigenfunction of $B_i$ with eigenvalue $\rho_i(n)$:
	\begin{align*}
		(B_if_l)(x)
		& = \int_\MM \sum_{m}\rho_i(m)f_m(x)f_m(y){f_l(y)}\mathrm{dV_g}(y)\\
		& = \sum_m \rho_i(m)f_m(x) \int_M f_m(y){f_l(y)} \mathrm{dV_g}(y)\\
		& = \sum_m \rho_i(m) f_m(x) \<f_m,f_l\>_{\MM}\\
		& = \sum_m \rho_i(m) f_m(x) \delta_{nm}\\
		& = \rho_i(l)f_l(x),
	\end{align*}
	where $\delta$ is the Kronecker delta and $\<\cdot,\cdot\>_\MM$ is the $L^2$ inner product on $\MM$ with $f_n$ being orthonormal basis. 
	
	Since $B_i$'s share the same eigenfunctions and hence commute, we have $d_l = \frac{\gamma^2_l}{\gamma^1_l}-1=\frac{\rho_2(l)}{\rho_1(l)}-1>-1$, so condition 2 holds by the definition of $\rho_i$. For condition 1, observe that $d_l = \frac{\rho_2(l)-\rho_1(l)}{\rho_1(l)}$, so \\
	$$ D\text{ is Hilbert--Schmidt} \Longleftrightarrow \sum_{l}d_l^2<\infty\Longleftrightarrow \sum_{l}\left|\frac{\rho_2(l)-\rho_1(l)}{\rho_1(l)}\right|^2<\infty.$$	
\end{proof}
\section{Proof of Theorem \changedrevieweronetwo{\ref{thm:mfd}}}\label{app:thm:mfd}

\begin{proof}
	We start with (A). First assume that $\nu_1=\nu_2=\nu$ and  $\sigma_1^2/C_{\nu,\alpha_1}=\sigma_2^2/C_{\nu,\alpha_2}$, then observe
	\begin{align*}
		\left|\frac{\rho_2(l)-\rho_1(l)}{\rho_1(l)}\right|&=\left|\frac{(\alpha_1^2+\lambda_l)^{\nu+d/2}}{(\alpha_2^2+\lambda_l)^{\nu+d/2}}-1\right|\\
		&\leq \left|(\alpha_1^2+\lambda_l)^{\nu+d/2}-(\alpha_2^2+\lambda_l)^{\nu+d/2}\right|/\lambda_l^{\nu+d/2}\\
		&\leq \left|((\alpha_1^2/\lambda_l+1)^{\nu+d/2}-((\alpha_2^2/\lambda_l+1)^{\nu+d/2}\right|.
	\end{align*}
	Note that $(1/x+1)^a=1+a /x+O(x^{-2})$ as $x\to\infty$, then \changedreviewertwo{when $l$ is sufficiently large so that $\lambda_l>0$},
	$$\left|(\alpha_1^2/\lambda_l+1)^{\nu+d/2}-(\alpha_2^2/\lambda_l+1)^{\nu+d/2}\right|\leq (\nu+d/2)(\alpha_1^2-\alpha_2^2)\lambda_l^{-1}+O(\lambda_l^{-2})=O(\lambda_l^{-1}).$$
	As a result,
	\begin{align*}
		\sum_{l}\left|\frac{\rho_2(l)-\rho_1(l)}{\rho_1(l)}\right|^2\lesssim \sum_l \lambda_l^{-2}.
	\end{align*}
    By Weyl's law (equation (4.1) in \cite{grebenkov2013geometrical}), $\lambda_l\sim l^{2/d}$, so we have $\lambda_l^{-2}\sim l^{-d/4}$ hence $-4/d<-1$ when $d\leq 3$. By the series test in Lemma \ref{lem:test}, $P_{1}\equiv P_{2}$.
	
	For the other direction, observe that 
	$$\left|\frac{\rho_2(l)}{\rho_1(l)}\right|={\left|\frac{\sigma_2^2C_{\nu_1,\alpha_1}(\alpha_1^2+\lambda_l)^{\nu_1+d/2}}{\sigma_1^2C_{\nu_2,\alpha_2}(\alpha_2^2+\lambda_l)^{\nu_2+d/2}}\right|}\to\begin{cases}\infty & \nu_1<\nu_2 \\
		1 & \nu_1=\nu_2\\
		0 &\nu_1>\nu_2.\end{cases}.$$
	As a result, if $\nu_1\neq \nu_2$, $\sum_{l}\left|\frac{\rho_2(l)}{\rho_1(l)}-1\right|\to\infty$ so $P_1\not\equiv P_2$ by the series test. 
	
	Then assume $\nu_1=\nu_2=\nu$ and $\sigma_1^2/C_{\nu,\alpha_1}\neq\sigma_2^2/C_{\nu,\alpha_2}$. Let $\sigma_0^2=\sigma_2^2\frac{C_{\nu,\alpha_1}}{C_{\nu,\alpha_2}}\neq \sigma_1^2$, then $$\sigma^2_0/C_{\nu,\alpha_1}=\sigma_2^2/C_{\nu,\alpha_2},$$
	so $k(\cdot;\sigma_0^2,\alpha_1)$ and $k(\cdot;\sigma_2^2,\alpha_2)$ define two equivalent measures, denoted by $P_{0}$ and $P_{2}$. Observe that 
	$$k(x,y;\sigma_1^2,\alpha_1)=\frac{\sigma_1^2}{\sigma_0^2}k(x,y;\sigma_0^2,\alpha_1),$$
\changedreviewertwo{then the corresponding spectral densities $\rho_0$ and $\rho_1$ only differ by a multiplicative scalar $\frac{\sigma_1^2}{\sigma_0^2}$ so $\sum_l\left|\frac{\rho_1(l)-\rho_0(l)}{\rho_1(l)}\right|^2=\sum_l \left|\frac{\sigma_1^2-\sigma_0^2}{\sigma_1^2}\right|^2=\infty$. So by Lemma \ref{lem:test},}	
\changedreviewerone{$P_{0}$ is orthogonal to $P_{1}$, so is $P_{2}$, which is equivalent to $P_0$}. Now we conclude that $P_1\equiv P_2$ if and only if $\sigma_1^2/C_{\nu,\alpha_1}=\sigma_2^2/C_{\nu,\alpha_2}$ and $\nu_1=\nu_2$.
	
	Then we show (B). As proved in (A), $P_1\not\equiv P_2$ if $\nu_1\neq \nu_2$ so we assume $\nu_1=\nu_2=\nu$.  Recall that $\lambda_n\to\infty$, so when $n$ is sufficiently large, $\lambda_n>\alpha^2$, then 
	\begin{align}
		&\left|\frac{\rho_2(l)-\rho_1(l)}{\rho_1(l)}\right|=\left|\frac{\sigma_2^2C_{\nu,\alpha_1}(\alpha_1^2+\lambda_l)^{\nu+d/2}}{\sigma_1^2C_{\nu,\alpha_2}(\alpha_2^2+\lambda_l)^{\nu+d/2}}-1\right| \nonumber\\
		&\geq \frac{\left|\sigma_2^2C_{\nu,\alpha_1}(\alpha_1^2+\lambda_l)^{\nu+d/2}-\sigma_1^2C_{\nu,\alpha_2}(\alpha_2^2+\lambda_l)^{\nu+d/2}\right|}{\sigma_1^2C_{\nu,\alpha_2}(2\lambda_l)^{\nu+d/2}}\nonumber\\
		&=2^{-\nu-d/2}\left|\frac{\sigma_2^2C_{\nu,\alpha_1}}{\sigma_1^2C_{\nu,\alpha_2}}\left(\alpha_1^2/\lambda_l+1\right)^{\nu+d/2}-(\alpha_2^2/\lambda_l+1)^{\nu+d/2}\right|\nonumber\\
		& = 2^{-\nu-d/2}\left|\frac{\sigma_2^2C_{\nu,\alpha_1}}{\sigma_1^2C_{\nu,\alpha_2}}-1+\left(\nu+\frac{d}{2}\right)\left(\frac{\sigma_2^2C_{\nu,\alpha_1}}{\sigma_1^2C_{\nu,\alpha_2}}\alpha_1^2-\alpha_2^2\right)\lambda_l^{-1}+O(\lambda_l^{-2})\right| \label{eqn:B}.
	\end{align}
	When $\sigma_1^2\neq \sigma_2^2$ or $\alpha_1\neq \alpha_2$, the constant term  $\frac{\sigma_2^2C_{\nu,\alpha_1}}{\sigma_1^2C_{\nu,\alpha_2}}-1$ and the linear coefficient $\frac{\sigma_2^2C_{\nu,\alpha_1}}{\sigma_1^2C_{\nu,\alpha_2}}\alpha_1^2-\alpha_2^2$ in Equation (\ref{eqn:B}) do not vanish at the same time hence $\left|\frac{\rho_2(l)-\rho_1(l)}{\rho_1(l)}\right|\gtrsim \lambda_{l}^{-1}$. Then $$\sum_{l}\left|\frac{\rho_2(l)-\rho_1(l)}{\rho_1(l)}\right|^2\gtrsim \sum_l \lambda_l^{-2}= \sum_{l} l^{-4/d}=\infty$$
	since $d\geq 4$. By the series test, $P_1\not\equiv P_2$. When $\sigma_1^2=\sigma_2^2$ and $\alpha_1= \alpha_2$, $P_1=P_2$ so $P_1\equiv P_2$, which finises the proof of (B).	
\end{proof}

	

\section{Proof of Theorem \ref{thm:rbf}}\label{app:thm:rbf}
\begin{proof}
	First assume $\alpha_1\neq \alpha_2$, or $\alpha_1<\alpha_2$ without loss of generality, then 
	\begin{align*}
		\left|\frac{\rho_2(l)-\rho_1(l)}{\rho_1(l)}\right|=\left|\frac{\sigma_2^2C_{\alpha_1}}{\sigma_1^2C_{\alpha_2}}e^{-\frac{\lambda_l}{2}\left(\frac{1}{\alpha_2^2}-\frac{1}{\alpha_1^2}\right)}-1\right|\to\infty
	\end{align*}
	since $\lambda_l\to \infty$. As a result,
	$$\sum_l \left|\frac{\rho_2(l)-\rho_1(l)}{\rho_1(l)}\right|^2=\infty.$$
	
	Then assume $\alpha_1=\alpha_2$ but $\sigma^2_1\neq \sigma_2^2$, similarly, 
	$$\sum_l \left|\frac{\rho_2(l)-\rho_1(l)}{\rho_1(l)}\right|^2=\sum_l \left(\frac{\sigma^2_2}{\sigma_1^2}-1\right)^2=\infty.$$
	Then the series test applies. 
\end{proof}

\section{Proof of Theorem \ref{thm:consist}}\label{app:thm:consist}
\begin{proof}
	Let $\sigma_1^2 = \frac{\sigma_0^2C_{\nu,\alpha_1}}{C_{\nu,\alpha_0}}$ so $P_0\equiv P_1$ by Theorem \ref{thm:mfd}. It suffices to show $\widehat{\sigma}_{1,n}^2\to\sigma_1^2$, $P_1$ a.s. Recall that $\widehat{\sigma}_{1,n}^2=\frac{Z_n^\T\Gamma_n^{-1}(\alpha_1)Z_n}{n}$ and  $Z_n\sim N(0,\sigma_1^2\Gamma_n(\alpha_1))$ under $P_1$, where $(\Gamma_n(\alpha))_{i,j}=\frac{1}{C_{\nu,\alpha}}\sum_{l=0}^\infty \left(2\nu\alpha^2+\lambda_l\right)^{-\nu-\frac{d}{2}}f_l(x_i)f_l(x_j)$. As a result, $\widehat{\sigma}_{1,n}^2=\sigma_1^2\frac{\chi_n^2}{n}\to \sigma_1^2$, $P_1$ a.s., as $n\to\infty$.
\end{proof}

\section{Proof of Lemma \ref{lem:prediction}}\label{app:lem:prediction}
\begin{proof}
The logic of the proof is similar to the proof of Theorem 1 and 2 in \cite{stein1993simple}. However, these two theorems are not directly applicable due to the discreteness of spectrum in our case. To be more specific, 
the key construction in the proof of \cite{stein1993simple} is the following. By the assumption, for any $\varepsilon>0$, there exists $M_\varepsilon>0$ such that $\sup_{m\geq M_\varepsilon}\left|\frac{\rho_1(m)}{c\rho_0(m)}-1\right|<\varepsilon$. We define $$\eta_\varepsilon(m)\coloneqq \begin{cases} \frac{1}{c}\rho_1(m) & m\leq M_\varepsilon\\
		\rho_0(m) & m>M_\varepsilon
	\end{cases}.$$
That is, $\eta_\varepsilon$ differs from $\rho_0$ only on a bounded subset of $\NN$.  Note that in \cite{stein1993simple}, the key step is to show $P_{\eta_\varepsilon}\equiv P_{\rho_0}$, and the rest of the proof will not rely on any special structure of the Euclidean domain anymore. That is, it suffices to show $P_{\eta_\varepsilon}\equiv P_{\rho_0}$, which is a direct consequence of the series test in Lemma 3. The rest of the proof of (i) naturally follows the proof of Theorem 1 in \cite{stein1993simple} while the proof of (ii) follows the proof of Theorem 2 in \cite{stein1993simple}, where $e(x_0,n,f_1)$ in \cite{stein1993simple} corresponds to $Z_0-\widehat{Z}_n(\rho_1)$ in our paper.
\end{proof}

\section{Proof of Theorem \ref{thm:prediction}}\label{app:thm:prediction}
\begin{proof}
	For $\sigma_1^2=\frac{\sigma_0^2C_{\nu,\alpha_1}}{C_{\nu,\alpha_0}}$, let $\rho_1$  and $\rho_0$ be the spectral density of the Gaussian process parametrised by $(\alpha_1,\sigma_1^2)$ and $(\alpha_0,\sigma_0^2)$ hence $\rho_1/\rho_0\to 1$. Then by (ii) in Lemma \ref{lem:prediction}, 
	$$\frac{\EE_{\sigma^2_1,\alpha_1}(\widehat{Z}_n(\alpha_1)-Z_0)^2}{\EE_{\sigma^2_0,\alpha_0}(\widehat{Z}_n(\alpha_1)-Z_0)^2}\to 1.$$
	Observe that 
	\begin{equation}\label{eqn:consist}
		\frac{\EE_{\widehat{\sigma}^2_{1,n},\alpha_1}(\widehat{Z}_n(\alpha_1)-Z_0)^2}{\EE_{\sigma^2_0,\alpha_0}(\widehat{Z}_n(\alpha_1)-Z_0)^2}   = \frac{\EE_{\widehat{\sigma}^2_{1,n},\alpha_1}(\widehat{Z}_n(\alpha_1)-Z_0)^2}{\EE_{\sigma^2_1,\alpha_1}(\widehat{Z}_n(\alpha_1)-Z_0)^2} \frac{\EE_{\sigma^2_1,\alpha_1}(\widehat{Z}_n(\alpha_1)-Z_0)^2}{\EE_{\sigma^2_0,\alpha_0}(\widehat{Z}_n(\alpha_1)-Z_0)^2}. 
	\end{equation}
	The second term in Equation (\ref{eqn:consist}) tends to $1$. For the first term, by the definition of $\widehat{Z}_n$, we obtain
	$$\EE_{\widehat{\sigma}^2_{1,n},\alpha_1}(\widehat{Z}_n(\alpha_1)-Z_0)^2=\widehat{\sigma}^2_{1,n}\left(1-\gamma_n(\alpha_1)^\T \Gamma_n(\alpha_1)^{-1}\gamma_n(\alpha_1)\right).$$
	Hence, the first term in Equation (\ref{eqn:consist}) is $\frac{\widehat{\sigma}_{1,n}^2}{\sigma_1^2}$. Similar to the proof of Theorem~\ref{thm:consist}, $\widehat{\sigma}_{1,n}^2=\sigma_1^2\frac{\chi_n^2}{n}\to \sigma_1^2$, $P_1\coloneqq P_{\sigma_1^2,\alpha_1}$ a.s. By Theorem~\ref{thm:mfd}, $P_0\equiv P_1$, so the left hand side of Equation (\ref{eqn:consist}) tends to $1$, $P_0$ a.s.
\end{proof}

\section{Proof of Theorem \ref{thm:spheres}}\label{app:thm:spheres}
\begin{proof}
$d$-dimensional spheres are compact	Riemannian manifolds. The eigenfunctions of the Laplace operator on $S^d$ are known as spherical harmonics, denoted by $S^l_m$, $m=0,1,\cdots$, $l=1,\cdots, t_d(m)$. The corresponding eigenvalues are $l(l+d-1)=O(l^2)$ with multiplicity \citep{muller2006spherical,frye2012spherical}
	$$\frac{2l+d-1}{l} \binom{l+d-2}{l-1} = O(l^{d-1}).$$ 
	So 1, 2, 3 follow directly from Theorem \ref{thm:mfd}, \ref{thm:consist} and \ref{thm:prediction} respectively.
\end{proof}

\section{Proof of Theorem {\ref{thm:cov2partial}}}\label{app:thm:cov2partial}
\begin{proof}
	First we reformulate the covariogram as
	\begin{align*}
		k(x,y)&=\frac{\sigma^2}{C_{\nu,\alpha}}\sum_{l=0}^\infty\left(\alpha^2+l(l+1)\right)^{-\nu-1}c_{l}\mathcal{L}_l(\cos(d_M(x,y))) = C\sum_{l=0}^\infty a_l(z)\;,
	\end{align*}
	where $C=\frac{\Gamma(3/2)\sigma^2}{2\pi^{3/2}C_{\nu,\alpha}}\changedreviewertwo{=\frac{4\pi\sigma^2}{\sum_{l} (2l+1)(\alpha^2+l(l+1))^{-\nu-1}}}$, $z=\cos(d_M(x,y))$ and
	$$a_l(z)=\left(\alpha^2+l(l+1)\right)^{-\nu-1}(2l+1)\mathcal{L}_l(z).$$
	 Observe that $\mathcal{L}_l(z)\in[-1,1]$. Therefore,
	\begin{align*}
		|a_l(z)|&\leq \left(\alpha^2+l(l+1)\right)^{-\nu-1}(2l+1).
	\end{align*}
	As a result,
	\begin{align*}
		|k^L(x,y)-k(x,y)|&\leq C\sum_{l=L+1}^\infty |a_l(z)|\leq \changedreviewertwo{C\sum_{l=L+1}^\infty\left(\alpha^2+l(l+1)\right)^{-\nu-1}(2l+1)}\\
		&\leq \changedreviewertwo{C}\sum_{l=L+1}^\infty \changedreviewertwo{(l^2)^{-\nu-1}(3l)}=3C\sum_{l=L+1}^\infty l^{-2\nu-1}\leq 3C\int_{L+1}^\infty t^{-2\nu-1}\mathrm{d}t\\
		& \changedreviewertwo{\leq \frac{3C}{2\nu} L^{-2\nu}=\frac{6\pi\sigma^2}{\nu\sum_{l} (2l+1)(\alpha^2+l(l+1))^{-\nu-1}}L^{-2\nu}}.
	\end{align*}
	That is, if the target approximation error is $\epsilon$, then we can truncate the infinite sum at
	$$L = \changedreviewertwo{\floor{\left(\frac{6\pi\sigma^2}{\epsilon\nu\sum_{l} (2l+1)(\alpha^2+l(l+1))^{-\nu-1}}\right)^{\frac{1}{2\nu}}}+1}.$$
	To prove positive definiteness, we first find the lower bound of the minimal eigenvalue of the covariance matrix $\Sigma\coloneqq\{k(x_i,x_j)\}_{ij}$, denoted by $\lambda_{\min}$. By Theorem 2.8 (i) in \cite{narcowich1998stability}, $$\lambda_{\min}\geq \xi_\rho(q)\coloneqq \Gamma_\rho(K)\left(1-\frac{\pi^3k(q)}{4q}\left(\frac{\sin(q/2)}{q/2}\right)^K\right),$$
	where 
	$$k(q)=\max\left\{\frac{8\pi^2}{q},\frac{25\pi}{2}\right\},~ K=\underset{m\in\NN}{\arg\min} \left\{m:1-\frac{\pi^3k(q)}{4q}\left(\frac{\sin(q/2)}{q/2}\right)^m>0\right\},$$
	and $\Gamma_\rho(K)$ is determined by the spectral density $\rho$, $K$ and the B-spline, see Equation (2.41) in \cite{narcowich1998stability} for further details (\changedreviewertwo{where $m=2$ in our setting}). Let the truncated covariance function be $\Sigma^L\coloneqq\{k^L(x_i,x_j)\}_{ij}$ with minimal eigenvalue $\lambda^N_{\min}$, then by the first half of the proof, 
	$$\lambda^L_{\min}\geq \lambda_{\min}-\|\Sigma-\Sigma^L\|\changedrevieweronetwo{\geq \xi_\rho(q)-n\|\Sigma-\Sigma^L\|_{\max}\geq\xi_\rho(q)-\frac{6\pi n\sigma^2 L^{-2\nu}}{\nu\sum_{l} (2l+1)(\alpha^2+l(l+1))^{-\nu-1}}}$$
    The second inequality follows from a matrix norm equivalence: $\|A\|_{\max}\leq\|A\|\leq n\|A\|_{\max}$ for any $n\times n$ matrix $A$. The first inequality relies on the fact that $\eig_{\min}(A)\geq \eig_{\min}(B)-\|A-B\|$ for symmetric matrices $A$ and $B$. Note that $\eig_{\min}(A)=\min_{\|x\|=1}x^\top Ax$ and let $x_0$ be the eigenvector of $A$ associated with the smallest eigenvalue, that is, $Ax_0 = \eig_{\min}(A)x_0$. By the same observation, $x_0^\top Bx_0\geq \eig_{\min}(B)$. Then,
$$\eig_{\min}(A) = x_0^\top Ax_0 = x_0^\top (B+A-B)x_0=x_0^\top B x_0 + x_0^\top (A-B)x_0\geq \eig_{\min}(B)+x_0^\top (A-B)x_0.$$
For the last term, since $\|A-B\|=\max_{\|x\|=1}|x^\top (A-B)x|$, we have
$|x_0^\top (A-B)x_0|\leq \|A-B\|$, hence $x_0^\top(A-B)x_0\geq -\|A-B\|$ as desired. \changedrevieweronetwo{Let $\xi_\rho(q)-\frac{6\pi n\sigma^2L^{-2\nu}}{\nu\sum_{l} (2l+1)(\alpha^2+l(l+1))^{-\nu-1}}>0$, we have}
    $$\changedrevieweronetwo{L> \left(\frac{6\pi n\sigma^2}{\nu\xi_\rho(q)\sum_l (2l+1)(\alpha^2+l(l+1))^{-\nu-1}}\right)^{\frac{1}{2\nu}}}$$
	
\end{proof}

\vskip 0.2in
\bibliography{ref}

\end{document}